\documentclass[10pt,leqno]{amsart}
\usepackage{graphicx}
\usepackage{indentfirst,csquotes}

\usepackage{amssymb,amsthm,amsmath}
\usepackage{xcolor,paralist,hyperref,titlesec,fancyhdr,etoolbox}
\newtheorem{theorem}{Theorem}
\newtheorem{definition}{Definition}

\newtheorem{lemma}{Lemma}
\newtheorem{proposition}{Proposition}

\titleformat{\section}[display]{\normalfont\huge\bfseries\centering}{\centering\chaptertitlename\thechapter}{10pt}{\Large}
\titlespacing*{\section}{0pt}{0ex}{0ex}

\hypersetup{ colorlinks=true, linkcolor=black, filecolor=black, urlcolor=black }

\usepackage{lipsum}

\begin{document}
\title[Inverse problem for the abstract diffusion-wave equation . . .]{Inverse problem for the abstract diffusion-wave equation with Caputo fractional derivative} 
\author[D.K.Durdiev]{D.K.Durdiev}
\author[H.H. Turdiev]{H.H. Turdiev}
\author[A.A. Rahmonov]{A.A. Rahmonov}
\date{\today}
\address{$ ^1$Bukhara Branch of Romanovskii Institute of Mathematics,
Uzbekistan Academy of Sciences, Bukhara, Uzbekistan}
\address{$^2$Bukhara State University, Bukhara, 705018 Uzbekistan}
\email{durdiev65@mail.ru}
\email{hturdiev@mail.ru}
\email{araxmonov@mail.ru}
\maketitle

\let\thefootnote\relax
\footnotetext{MSC2020: 34A08, 34K10, 34K29, 34K37, 35R11, 35R30} 

\begin{abstract}
In this work, we study  the inverse problem of determining a potential coefficient in an abstract wave equation that includes a lower-order term. The equation incorporates a time-fractional derivative in the Caputo sense, as well as a fractional power of an abstract operator defined on a Hilbert space.
Using the Fourier decomposition method, we analyze the solvability of the direct problem. Leveraging the properties of the solution to the direct problem, we conduct an  examination of the inverse problem. By applying the fixed point theorem within a suitable Banach space, we derive results concerning the local existence, uniqueness, and stability of the solution.
\end{abstract} 

\bigskip

{\bf Keywords:} Caputo fractional derivative, power of operator, Mittag-Leffler function, inverse problem, integral equation, Fourier series, principle of contraction mapping

In recent years, fractional diffusion and diffusion-wave equations have garnered increasing attention not only from mathematicians but also from researchers across a wide range of scientific disciplines, including engineering, biology, physics, chemistry, biochemistry, medicine, and finance. Fractional differential equations have found applications in various areas of modern science, such as dynamic processes in porous media, diffusion transport, electrical networks, control theory of dynamic systems, viscoelasticity, and more. Time-space fractional diffusion and diffusion-wave equations are formulated by replacing the conventional time and space derivatives with their fractional counterparts \cite{forhtur1,forhtur2,forhtur3,forhtur4,forhtur5}.

In the works \cite{forhtur6,forhtur7,forhtur8,forhtur9}, numerous authors have studied the unique solvability of fractional initial-boundary value problems for diffusion-wave equations involving second-order elliptic differential operators. These investigations consider various initial and boundary conditions and employ different solving  approaches.

  In the works \cite{forhtur18,forhtur19,forhtur20,forhtur21,forhtur22}, the problem of determining the time-dependent coefficient in the fractional-order diffusion-wave equation was extensively investigated. These studies focused on equations involving a second-order elliptic differential operator under various boundary conditions. In all of these contributions, the authors established theorems guaranteeing the existence and uniqueness of solutions to the inverse problem of identifying the zero-order coefficient.

In this article, we'll consider an abstract fractional diffusion-wave equation in the Hilbert space $H$:
\begin{equation}\label{equation1}
\partial_{0+,t}^{\alpha}u(t)+A^{\beta}u(t)+q(t)u(t)=f(t),\,\, t\in (0,T),
 \end{equation}
 with conditions
\begin{equation}\label{equation2}
u(x,t)\big|_{t=0}=\varphi,\quad u_t(t)\big|_{t=0}=\psi, \end{equation}
where $\partial_{0+,t}^{\alpha}$ is the Caputo fractional derivative of order $1<\alpha < 2$ in the time variable $t$ (see definition 1, 2 in preliminaries),  $0<\beta<1,$  $f(t) \in C((0, T) ; H)$ is a given function, $\varphi,\psi $  are  known elements of $H$, $H$ is a separable Hilbert space equipped with the inner product $(\cdot, \cdot)$ and the norm $\|\cdot\|$, $A: H \rightarrow H$ is  an arbitrary self-adjoint, a densely defined,  positive, unbounded operator in $H$ with the domain of definition $D(A) \subset H$. Assume that operator $A$ has a set of positive eigenvalues $\left\{\lambda_k\right\}$ and a complete set of orthonormal eigenfunctions $\left\{e_k\right\}$. Let the sequence of eigenvalues of the operator $A$ have no finite limit point. Using renumbering of eigenvalues, we can number them in a non-decreasing manner, i.e., $0<\lambda_1 \leq \lambda_2 \leq \cdots \rightarrow+\infty$.

In what follows, we refer to the problem defined by (\ref{equation1})–(\ref{equation2}) as the direct problem.

\begin{definition}\label{definition1}
A function $u(t) \in C([0,T];D(A^{\beta}))\cap C^1((0, T];H)$ with the properties $D_t^\alpha u(t)\in C((0,T];H),$ $A^{\beta} u(t) \in C((0,T];H)$ and satisfying conditions (\ref{equation1}),(\ref{equation2}), is called the  solution of direct problem.
\end{definition}

We will prove the existence of a unique solution to the direct problem, and also obtain some results about its regularity. In the main part of the article, based on the direct problem, the following inverse problem for finding the coefficient $q(t)$ in equation (\ref{equation1}) is considered.

\textit{Inverse problem.} Given $\alpha,\,\, \beta,\,\, f(t)$ and $\varphi,\,\, \psi$, find a function $ q(t)$ satisfying the problem (\ref{equation1})-(\ref{equation2}) and the additional condition
\begin{equation}\label{equation3}
\Phi[u(t)]=\mu(t), \quad 0 \leq t \leq T,
\end{equation}
where $\mu:[0,T] \rightarrow \mathbb{R}$ is a given function, $\Phi: D(\Phi) \subset H \rightarrow \mathbb{R}$ is a known linear bounded functional.

\begin{definition}\label{definition11111}
The problem of determining the function $q(t)$ from equation (\ref{equation1}) using the additional condition (\ref{equation3}) imposed on the solution of the direct problem (\ref{equation1})-(\ref{equation2}) is called the inverse problem.
\end{definition}

Problems for equation (\ref{equation1}) have a classical character and a long history, and are currently being thoroughly studied. The main source of difficulty in the analysis of problem (\ref{equation1})-(\ref{equation2}) and in the design
of efficient solution techniques is the nonlocality of both the fractional time derivative and the fractional space operator. Problems for fractional time derivatives and fractional space operator equations have been studied by many researchers.   When $\beta=1$,  in the paper \cite{forhtur23}, the authors study a class of a fractional differential equations in Hilbert spaces the following type
$$
\partial_{0+,t}^{\alpha}u(t)+A u(t)=0, \quad t \in[0, T],
$$
where $\alpha \in(1,2)$, $A$ is linear self-adjoint positive operator. The authors introduce notions of weak and strong solutions for these equation and present conditions under which there exist solutions.
Existence results were obtained for weak and strong solutions.

In the article \cite{forhtur24},  the authors studied the Cauchy problem for a fractional evolution equation of the type
$$
\partial_{0+,t}^{\alpha} u(t)=A u(t)+f(t), \quad t \in(0, a],
$$
where $a>0$ and $\alpha>0,$  $A: D(A) \subseteq X \rightarrow X$ is a closed linear operator defined in a Banach space $X$, the function $f$ belongs to an appropriate space of $X$-valued functions. In this case, the differentiability of mild solution for the fractional-order abstract Cauchy problems was investigated.  In the first part, the existence of classical solutions of the homogeneous Cauchy problem is considered, and in the second part, the existence of classical solutions of the non-homogeneous abstract Cauchy problem is investigated.

In \cite{forhtur26}, the initial-boundary value problem for the time-fractional wave equation, as well as the associated inverse problem concerning the determination of the zero-order term, were investigated. First, the direct problem was studied under a Robin-type boundary condition. Utilizing the properties of the direct problem, the inverse problem was subsequently addressed. The local solvability of the inverse problem was established, and its stability was demonstrated. Moreover, in \cite{forhtur27}, the inverse problem related to the identification of coefficients in the abstract fractional diffusion equation was rigorously analyzed. The study established results concerning the local existence, uniqueness, and stability of the inverse problem’s solution under appropriate assumptions.

In this paper, the following assumptions are made:

(I1) $\varphi \in D\left(A^{2\beta}\right),\,\,\psi  \in D\left(A^{2\beta}\right),\,\, f \in C\left([0, T] ; D\left(A^{\beta}\right)\right)$;

(I2) $\mu(t) \in C^1[0, T]$ and satisfy the conditions
$|\mu(t)|\geq \mu_0>0$ where $\mu_0$ is given number;

(I3) $\mu(0)=\Phi[\varphi]$, $\mu'(0)=\Phi[\psi]$;

(I4) $\Phi:\left\{\Phi\left[e_k\right]\right\} \in l^2(\mathbb{N})$, where $\mathbb{N}$ denotes the set of natural numbers.

In the next section, we present essential definitions and assertions required for the formulation and proofs of the results.

\section*{Preliminaries}

\begin{definition}\label{definition1}(see \cite[pp. 69-76]{forhtur1})   The Riemann-Liouville fractional integral of order $0 <\alpha< 1$ for an integrable
function $h(t)\in AC[0,T]$ is defined by
\begin{equation*}
I^{\alpha}_{0+,t}h(t)=\frac{1}{\Gamma(\alpha)}\int\limits_{0}^t (t-\tau)^{\alpha-1} h(\tau) d\tau,\,\, t>0.
\end{equation*}
\end{definition}
\begin{definition}\label{definition2}(see \cite[pp. 69-76]{forhtur1})
 The Caputo fractional derivative of order $n-1 <\alpha< n$ of the functions $h(t)\in AC^{n}[0,T]$  is defined by
\begin{equation*}
\left(\partial_{t}^{\alpha}h\right)(t)=\frac{1}{\Gamma(n-\alpha)}  \int_0^t\frac{h^{(n)}(\tau)}{(t-\tau)^{\alpha-n+1}}d\tau,\,\, t>0.
\end{equation*}
\end{definition}

{\bf Two parameter Mittag-Leffler (M-L) function.} The two parameter M-L function $E_{\alpha,\beta}(z)$ is defined by the following series:
\begin{equation*}
E_{\alpha,\beta}(z)=\sum\limits_{k=0}^{\infty}\frac{z^k}{\Gamma(\alpha k+\beta)},
\end{equation*}
where $\alpha,\beta,z\in\mathbb{C}$ with $\mathfrak{R}(\alpha)>0$, $\mathfrak{R}(\alpha)$ denotes the real part of the complex number $\alpha$, $\Gamma(\cdot)$ is Euler's Gamma function. The Mittag-Leffler function has been extensively studied by numerous researchers, leading to various generalizations and applications. A particularly noteworthy contribution that consolidates many significant results on this function is the book by Kilbas et al (see  \cite[pp. 42-44]{forhtur1}).

The case  $\beta=1$ reduces to the Mittag-Leffler  function of single parameter, i.e.
\begin{equation*}
E_{\alpha,1}(z)=\sum\limits_{k=0}^{\infty}\frac{z^k}{\Gamma(\alpha k+1)}.
\end{equation*}

\begin{proposition}\label{proposition1}
 Let $0<\alpha<2$ and $\beta\in\mathbb{R}$ be arbitrary. We suppose that $\kappa$ is such that $\pi\alpha/2<\kappa<\min\{\pi,\pi\alpha\}$. Then there exists a constant $C=C(\alpha,\beta,\kappa)>0$ such that
\begin{equation*}
\left|E_{\alpha,\beta}(z)\right|\leq\frac{C}{1+|z|},\quad \kappa\leq|\mbox{arg}(z)|\leq\pi.
\end{equation*}
\end{proposition}
For the proof, we refer to (\cite[pp. 40-45]{forhtur1}.

Let $\sigma$ be an arbitrary real number. We introduce the power of operator $A$, acting in $H$ according to the rule (see, \cite[p.69]{forhtur28},\cite[p.57]{forhtur29})
$$
A^\sigma h=\sum_{k=1}^{\infty} \lambda_k^\sigma h_k e_k,
$$
where $h_k$ is the Fourier coefficients of a function $h \in H: h_k=\left(h, e_k\right)$. The domain of this operator has the form
$$
D\left(A^\sigma\right)=\left\{h \in H: \sum_{k=1}^{\infty} \lambda_k^{2 \sigma}\left(h, e_k\right)^2<\infty\right\}.
$$
It immediately follows from this definition that $D\left(A^{\sigma_1}\right) \subseteq D\left(A^{\sigma_2}\right)$ for any $\sigma_1 \geq \sigma_2$.
On the set $D\left(A^\sigma\right)$, we define the inner product
$$
(h, g)_\sigma=\sum_{k=1}^{\infty} \lambda_k^{2 \sigma} h_k \bar{g}_k=\left(A^\sigma h, A^\sigma g\right).
$$
For elements of $D\left(A^\sigma\right)$ we introduce the norm
$$
\|h\|_{D\left(A^\sigma\right)}^2=(h, h)_\sigma=\sum_{k=1}^{\infty} \lambda_k^{2 \sigma} |h_k|^2=\left(A^\sigma h, A^\sigma h\right)=\left\|A^\sigma h\right\|^2.
$$

\begin{lemma}\label{lemma1} ([30], p.189) Suppose $b\geq 0,\,\, \alpha>0,\,\, \gamma>0,\,\, \alpha+\gamma>1$ and $a(t)$ nonnegative function locally integrable on  $0\leq t<T$ and suppose $t^{\gamma-1}u(t)$ is nonnegative and locally integrable on $0\leq t<T$ with
\begin{equation*}
u(t)\leq a(t)+b\int_{0}^t(t-s)^{\alpha-1}s^{\gamma-1}u(s)ds
\end{equation*}
a.e. in $(0,T)$; then
\begin{equation*}
u(t)\leq a(t)Z_{\alpha,\gamma}\left(\left(b\Gamma(\alpha)\right)^{\frac{1}{\alpha+\gamma-1}}t\right),
\end{equation*}
where $Z_{\alpha,\gamma}(t)=\sum\limits_{m=0}^{\infty}c_{m}t^{m(\alpha+\gamma-1)},\,\,\, c_{0}=1,\,\,\frac{c_{m+1}}{c_{m}}=\frac{\Gamma(m(\alpha+\gamma-1)+\gamma)}{\Gamma(m(\alpha+\gamma-1)+\alpha+\gamma)}
$ for $m\geq 0$. As $t\rightarrow +\infty$
$$Z_{\alpha,\gamma}(t)=O\left(t^{\frac{1}{2}\frac{\alpha+\gamma-1}{\alpha}-\gamma}
\exp\left(\frac{\alpha+\gamma-1}{\alpha}t^{\frac{\alpha+\gamma-1}{\alpha}}\right)\right).$$
\end{lemma}

\section*{Solvability of  direct problem}

\begin{theorem}\label{thm1}
For any $\varphi, \psi \in H$ and $f(t) \in C([0, T] ; H)$ problem (\ref{equation1})-(\ref{equation3}) has a unique solution and this solution has the form
$$
u(t)=\sum_{j=1}^{\infty}\Bigg[\varphi_j E_{\alpha, \alpha}\left(-\lambda_j t^\alpha\right)+\psi_j t E_{\alpha, \alpha-1}\left(-\lambda_j t^\alpha\right)
$$$$
+\int_0^t f_j(t-\xi) \xi^{\alpha-1} E_{\alpha, \alpha}\left(-\lambda_j \xi^\alpha\right) d \xi\Bigg] v_j,
$$
 where the series converges in $H, f_j(t), \varphi_j$ and $\psi_j$ are corresponding Fourier coefficients.
\end{theorem}

To solve the direct problem, we divide it into two
auxiliary problems:
\begin{equation}\label{equation4}
\left\{\begin{array}{l}
\partial_{0+,t}^{\alpha}v_1(t)+A^{\beta}v_1(t)+q(t)v_1(t)=f(t), \quad 0<t \leq T, \\
v_1(t)\big|_{t=0}=0,\quad v_{1t}(t)\big|_{t=0}=0,
\end{array}\right.
\end{equation}
and
\begin{equation}\label{equation5}
\left\{\begin{array}{l}
\partial_{0+,t}^{\alpha}v_2(t)+A^{\beta}v_2(t)+q(t)v_2(t)=0, \quad 0<t \leq T, \\
v_2(t)\big|_{t=0}=\varphi,\quad v_{2t}(t)\big|_{t=0}=\psi.
\end{array}\right.
\end{equation}

First of all, we will investigate problem (\ref{equation4}). For this, we use the method given in \cite{forhtur31}. We define the operator valued function $Y(t)$ by
\begin{equation}\label{equation6}
Y(t)h=\sum_{k=1}^{\infty}\left(h, e_k\right) t^{\alpha-1} E_{\alpha, \alpha}\left(-\lambda_k^{\beta} t^\alpha\right) e_k, \quad h \in H,\,\, t>0,
\end{equation}
and the operator $A^{\beta}Y(t)$ is as follows
\begin{equation}\label{equation7}
A^{\beta}Y(t)h=\sum_{k=1}^{\infty}\lambda_k^{\beta}\left(h, e_k\right) t^{\alpha-1} E_{\alpha, \alpha}\left(-\lambda_k^{\beta} t^\alpha\right) e_k, \quad h \in H,\,\, t>0.
\end{equation}

From the estimate of the Mittag-Leffler function given in Proposition \ref{proposition1} and (\ref{equation7}) we obtain
\begin{equation*}
\|A^{\beta}Y(t)h\|^2=\sum_{k=1}^{\infty}\left|\lambda_k^{\beta}\left(h, e_k\right) t^{\alpha-1} E_{\alpha, \alpha}\left(-\lambda_k^{\beta} t^\alpha\right)\right|^2
\end{equation*}
\begin{equation*}
\leq C^2 \sum_{k=1}^{\infty}\left|t^{\alpha -1} \lambda_k^{\beta}\left(h, e_k\right)\right|^2=C^2 t^{2(\alpha -1)} \sum_{k=1}^{\infty}\left|\lambda_k^{\beta}\left(h, e_k\right)\right|^2=C^2 t^{2(\alpha -1)} \left\|h\right\|_{D(A^\beta)}^2.
\end{equation*}
Thus, we have
\begin{equation}\label{equation8}
\|A^{\beta}Y(t)h\|\leq C t^{\alpha -1} \left\|h\right\|_{D(A^\beta)}, \,\, h\in H, \,\,t\geq 0.
\end{equation}

For problem (\ref{equation4}) we have the following statement.
\begin{theorem}\label{theorem222}
Let $q(t) \in C[0, T]$ and $f(t) \in C\left([0, T] ; D\left(A^{\beta}\right)\right)$ for some $\beta \in(0,1)$. Then problem (\ref{equation4}) has a unique solution and it satisfies the following integral equation
\begin{equation}\label{equation9}
v_1(t)=\int_0^t Y(t-s) f(s) d s-\int_0^t Y(t-s) q(s) v_1(s) d s.
\end{equation}
Moreover, there is a constant $c>0$ such that the following coercive type inequality holds:
\begin{equation}\label{equation10}
\left\|\partial_{0+,t}^{\alpha} v_1\right\|_{C([0,T];H)}+\|A^{\beta} v_1\|_{C([0,T];H)} \leq C\|f\|_{C\left([0,T];D\left(A^{\beta}\right)\right)}.
\end{equation}
\end{theorem}

\begin{proof}

First of all, let's consider the auxiliary abstract problem. Let's consider the following  Cauchy problem
\begin{equation}\label{equation11}
\left\{\begin{array}{l}
\partial_{0+,t}^{\alpha}w(t)+A^{\beta}w(t)=F(t), \quad 0<t \leq T, \\
w(t)\big|_{t=0}=0,\quad w_{t}(t)\big|_{t=0}=0.
\end{array}\right.
\end{equation}

For the solution of problem (\ref{equation11}), the following lemma holds.
\begin{lemma}\label{lemma3}
Let $F \in C\left([0,T] ; D\left(A^{\beta}\right)\right),\,\, \beta\in (0,1)$. Then abstract Cauchy problem (\ref{equation11}) has a unique solution and this solution has the representation
\begin{equation}\label{equation12}
w(t)=\int_0^t Y(t-s) F(s) d s.
\end{equation}
Moreover, the solution of the abstract Cauchy problem (\ref{equation11}) satisfies the following inequality:
\begin{equation}\label{equation13}
\|w\|_{C([0, T] ; D(A^{\beta}))} \leq c\|F\|_{C\left([0, T] ; D\left(A^{\beta}\right)\right)},
\end{equation}
where $c=c (\alpha, T) >0$ is a constant.
\end{lemma}

\begin{proof}
Let us assume that a solution to problem (\ref{equation11}) exists.  Then, due to the completeness of the system of eigenfunctions ${e_k}$ of the operator $A$, this solution will have the form
 \begin{equation}\label{equation14}
w(t)=\sum\limits_{k=1}^{\infty}T_k(t)e_k,
\end{equation}
where $T_k(t)$ are unknown coefficients.
If we multiply both sides of this equality scalarly by ${e_k}$, then from the orthonormality of the system of eigenfunctions ${e_k}$ we obtain the equalities $T_k(t) = (w(t),e_k)$. We substitute (\ref{equation14}) to problem  (\ref{equation11}) and we obtain the following problem:
\begin{equation}\label{equation15}
\left\{\begin{array}{l}
\partial_{0+,t}^{\alpha}T_k(t)+\lambda_k^{\beta}T_k(t)=F_k(t),\,\, 0<t \leq T, \\
T_k(t)\big|_{t=0}=0,\quad T_k'(t)\big|_{t=0}=0,
\end{array}\right.
\end{equation}
where $F_k(t)$ are the Fourier coefficients of the functions $F(t)$. The theory of fractional ordinary differential equations \cite{forhtur1} gives a unique function $T_k(t)$ satisfying problem (\ref{equation15}).  By separating the variables and Laplace transform of the Mittag–Leffler function, the formal solution of the problem (\ref{equation15}) can be expressed as:
\begin{equation}\label{equation16} w(t)=\sum_{k=1}^{\infty}\left[\int\limits_0^t(t-s)^{\alpha-1} E_{\alpha,\alpha}\left(-\lambda_k^{\beta}(t-s)^\alpha\right)\left(F(s),e_k\right) d s\right] e_k.
\end{equation}
Then it follows that for $F \in C\left([0, T] ; D\left(A^{\beta}\right)\right)$, we have
$$
A^{\beta} w(t)=\int\limits_0^t A^{\beta} Y(t-s) F(s) d s.
$$
Furthermore, for $\beta \in(0,1)$, by  Proposition \ref{proposition1}, we obtain
$$
\|A^{\beta} w\| \leq \int_0^t\|A^{\beta} Y(t-s) F(s)\| d s
\leq C \int_0^t(t-s)^{\alpha -1}\|F(s)\|_{D\left(A^{\beta}\right)} d s
$$$$
\leq C \frac{t^{\alpha }}{\alpha } \max _{0 \leq s \leq t}\|F(s)\|_{D\left(A^{\beta}\right)}.
$$
Therefore we have
$$
\|A^{\beta} w\|_{C([0, T] ; H)} \leq C_1 T^{\alpha}\|F\|_{C\left([0, T] ; D\left(A^{\beta}\right)\right)}
$$
and in particular $w(t) \in C([0, T] ; D(A^{\beta}))$.

Furthermore, by the original equation $\partial_{0+,t}^{\alpha} w(t)=-A^{\beta} w(t)+F(t), t>0$, we have $\partial_{0+,t}^{\alpha} w(t) \in$ $C([0, T] ; H)$ and
$$
\left\|\partial_{0+,t}^{\alpha} w\right\|_{C([0, T] ; H)} \leq C_2\|F\|_{C\left([0, T] ; D\left(A^{\beta}\right)\right)}.
$$

Thus, we have completed the rationale that (\ref{equation12}) is a solution to problem (\ref{equation11}). Let's move on to the proof of the uniqueness of the solution to problem (\ref{equation11}). Suppose that problem (\ref{equation11}) has two solutions $w_1(t)$ and $w_2(t)$. Our goal is to prove that $w(t)=w_1(t)-$ $w_2(t) \equiv 0$. Since the problem is linear, for $w(t)$, then, we have the following homogeneous problem:
\begin{equation}\label{equation17}
\left\{\begin{array}{l}
\partial_{0+,t}^{\alpha}w(t)+A^{\beta}w(t)=0, \quad 0<t \leq T, \\
w(t)\big|_{t=0}=0,\quad w_{t}(t)\big|_{t=0}=0.
\end{array}\right.
\end{equation}

Let $w_k(t)=\left(w(t), e_k\right)$. From (\ref{equation17}) it follows that for any $k \in \mathbb{N}$ :
$$
\partial_{0+,t}^{\alpha} w_k(t)=\left(\partial_{0+,t}^{\alpha} w(t), e_k\right)=-\left(A^{\beta} w(t), e_k\right)=-\left(w(t), A^{\beta} e_k\right)=-\lambda_k^{\beta} w_k(t) .
$$

Thus, we have obtained the following Cauchy problem for $w_k(t)$ :
\begin{equation*}
\left\{\begin{array}{l}
\partial_{0+,t}^{\alpha}w_k(t)+\lambda_k^{\beta}w_k(t)=0,\,\, 0<t \leq T,\\
w_k(t)\big|_{t=0}=0,\quad w_k'(t)\big|_{t=0}=0.
\end{array}\right.
\end{equation*}
This problem has a unique solution (see \cite[p.231]{forhtur1}). Therefore, we have $w_k(t)=0$ for $t \geq 0$ and for all $k \in \mathbb{N}$. Then by Parseval's equality we obtain $w(t)=0$ for all $t \geq 0$. The Lemma \ref{lemma3} is proved.
\end{proof}

Here we define the map $\mathcal{G}: C\left([0, T] ; D\left(A^{\beta}\right)\right) \rightarrow C\left([0, T] ; D\left(A^{\beta}\right)\right)$ by
\begin{equation}\label{equation18}
\mathcal{G}(F)(t)=\int_0^t Y(t-s) F(s) d s,
\end{equation}
where $F(t) \in C([0, T] ; D(A^{\beta})).$\\
Therefore, according to lemma \ref{lemma3}, we have
$$
\|\mathcal{G}(F)\|_{C([0, T] ; D(A^{\beta}))} \leq c\|F\|_{C\left([0, T] ; D\left(A^{\beta}\right)\right)}.
$$

Now, let us turn to prove theorem \ref{theorem222}.
Let us write the Cauchy problem (\ref{equation4}) as follows
\begin{equation}\label{equation19}
\left\{\begin{array}{l}
\partial_{0+,t}^{\alpha}v_1(t)+A^{\beta}v_1(t)=f(t)-q(t)v_1(t), \quad 0<t \leq T, \\
v_1(t)\big|_{t=0}=0,\quad v_{1t}(t)\big|_{t=0}=0.
\end{array}\right.
\end{equation}

If we take $F (t) =f (t) -q (t) v_1 (t) $, then according to the solution (\ref{equation12}) of problem (\ref{equation11}), we obtain an integral equation equivalent to the problem (\ref{equation19}) with respect to $v_1$:
\begin{equation}\label{equation20}
v_1(t)=\int_0^t Y(t-s) f(s) d s-\int_0^t Y(t-s) q(s) v_1(s) d s.
\end{equation}

For integral equation (\ref{equation20}), we apply the theorem about a fixed point. So we will look for a fixed point of the operator $\mathcal{O}: C([0, T] ; D(A^{\beta})) \rightarrow C([0, T] ; D(A^{\beta}))$ defined by
$$
\mathcal{O}(v_1)(t)=-(\mathcal{G}(q))(v_1)(t)+\mathcal{G}(f)(t),\,\, t\in (0,T),
$$
for $v_1 \in C([0, T] ; D(A^{\beta}))$.

We apply the operator $\mathcal{O}$ to $v_1(t)$ twice in a row
$$
\mathcal{O}^2(v_1)(t)=(\mathcal{G}(q))^2(v_1)(t)-(\mathcal{G}(q))\mathcal{G}(f)(t)+\mathcal{G}(f)(t),\,\, t\in (0,T).
$$
By induction, we have
$$
\mathcal{O}^n(v_1)(t)=(-1)^n(\mathcal{G}(q))^n(v_1)(t)+\sum\limits_{i=0}^{n-1}(-1)^i(\mathcal{G}(q))^i\mathcal{G}(f)(t),\,\, t\in (0,T).
$$
Here we denote $(\mathcal{G}(q)(t))^0=I$ is identity operator.\\
By  $q(t) \in C[0, T]$   and   $v \in C([0, T] ; D(A^{\beta}))$, we known that   $q v \in C([0, T] ; D(A^{\beta}))$   and
\begin{equation}\label{equation21}
 \|q(s) v(s)\|_{D(A^{\beta})} \leq \|q\|_{C[0,T]}\|v(s)\|_{D(A^{\beta})}.
\end{equation}

By (\ref{equation8}), for $v_1 \in C([0, T] ; D(A^\beta))$ and all $t \in[0, T]$, we have
$$
\|(\mathcal{G}(q))(v_1)(t)\|_{D(A^{\beta})} =\left\|\int_0^t Y(t-s) q(s) A^{\beta} v_1(s) d s\right\|
$$$$
\leq \|q\|_{C[0,T]} \int_0^t(t-s)^{\alpha-1}\|v_1(s)\|_{D(A^{\beta})} d s \leq C T^{\alpha} \|q\|_{C[0,T]}\|v_1\|_{\mathcal{C}([0,T];D(A^\beta))}.
$$
Then,  for every $v_1 \in C([0, T] ; D(A^{\beta}))$, we have $(\mathcal{G}(q))(v_1) \in C([0, T] ; D(A^{\beta}))$ and the estimate
$$
\|(\mathcal{G}(q))(v_1)(t)\|_{C([0, T] ; D(A^{\beta}))} \leq C T^{\alpha} \|q\|_{C[0,T]}\|v_1\|_{\mathcal{C}([0,T];D(A^\beta))}.
$$
Thus we can see that $(\mathcal{G}(q))$ maps $C([0, T] ; D(A^{\beta}))$ into itself. Combining $\mathcal{G} f \in C([0, T] ; D(A^{\beta}))$, we get that the operator $\mathcal{G}$ also maps $C([0, T] ; D(A^{\beta}))$ into itself.

Repeating the same calculations for $v_1\in C([0,T];D (A^{\beta}))$, we obtain
$$
\left\|(\mathcal{G}(q))^2(v_1)(t)\right\|_{D(A^{\beta})} \leq C \int_0^t(t-s)^{\alpha -1}\|(\mathcal{G}(q(s)))(v_1)(s)\|_{D(A^{\beta})} d s
$$$$
\leq C^2 \int_0^t(t-s)^{\alpha  -1}\left(\int_0^s(s-\tau)^{\alpha  -1}\|v_1(\tau)\|_{D(A^{\beta})} d \tau\right) d s
$$$$
=C^2 \int_0^t\left(\int_\tau^t(t-s)^{\alpha -1}(s-\tau)^{\alpha -1} d s\right)\|v_1(\tau)\|_{D(A^{\beta})} d \tau
$$$$
=\frac{(C\Gamma(\alpha ))^2}{\Gamma(2 \alpha )} \int_0^t(t-\tau)^{2 \alpha -1}\|v_1(\tau)\|_{D(A^{\beta})} d \tau
$$
or
$$
\left\|(\mathcal{G}(q))^2(v)(t)\right\|_{D(A^{\beta})} \leq \frac{(CT^{\alpha}\|q\|_{C[0,T]}\Gamma(\alpha ))^2}{\Gamma(2 \alpha +1)}\|v_1\|_{C([0,T]; D(A^{\beta}))}.
$$
By induction, for all $v_1 \in C([0,T];D(A^{\beta}))$, we have
\begin{equation}\label{equation22}
\left\|\mathcal{G}^n w(t)\right\|_{D(A^{\beta})}   \leq \frac{(CT^{\alpha}\|q\|_{C[0,T]}\Gamma(\alpha ))^n}{\Gamma(n \alpha +1)}\|v_1\|_{C([0, T];D(A^{\beta}))}, \quad t \in[0, T].
\end{equation}
Therefore, we have $(\mathcal{G}(q))^n(v) \in$ $C([0, T];D(A^{\beta}))$. Therefore, for $v^1_1, v^2_1 \in C([0,T];D(A^{\beta}))$, we obtain
$$
\begin{gathered}
\left\|\mathcal{O}^n\left(v_1\right)-\mathcal{O}^n\left(v_2\right)\right\|_{C([0, T] ; D(A^{\beta}))}=\left\|(\mathcal{G}(q))^n\left(v_1-v_2\right)\right\|_{C([0, T] ; D(A^{\beta}))} \\
\leq \frac{(CT^{\alpha}\|q\|_{C[0,T]}\Gamma(\alpha ))^n}{\Gamma(n \alpha +1)}\left\|v^1_1-v^2_1\right\|_{C([0, T] ; D(A^{\beta}))}.
\end{gathered}
$$
It is easy to verify $\frac{(CT^{\alpha}\|q\|_{C[0,T]}\Gamma(\alpha ))^n}{\Gamma(n \alpha +1)} \rightarrow 0$ as $n \rightarrow \infty$, since $\Gamma(n+1) \sim \frac{n^n}{e^n} \sqrt{2 \pi n}$. Therefore, we have $\frac{(CT^{\alpha}\|q\|_{C[0,T]}\Gamma(\alpha ))^n}{\Gamma(n \alpha +1)}<1$ for sufficiently large $n \in \mathbb{N}$. Thus, the operator $\mathcal{O}^n$ is a contraction mapping from $C([0,T];D(A^{\beta}))$ into itself. Hence the mapping $\mathcal{G}^n$ admits unique fixed solution $v_1 \in C([0,T];D(A^{\beta}))$, that is, $\mathcal{G}^n(v_1)=v_1$. Since $\mathcal{G}^{n+1}(v_1)=\mathcal{G}^n(\mathcal{G}(v_1))=\mathcal{G}(v_1)$, the point $\mathcal{G}(v_1)$ is also a fixed point of the mapping $\mathcal{G}^n$. By the uniqueness of the fixed point of $\mathcal{G}^n$, we get $(\mathcal{G}(q))(v_1)+\mathcal{G}(f)=\mathcal{O}(v_1)=v_1$, that is, the equation (3.3.21) has a unique solution $v_1 \in C([0,T];D(A^{\beta}))$. Moreover, for any $n \in \mathbb{N}$, we have
$$
v_1=\mathcal{O}(v_1)=\mathcal{O}^n(v_1)=(-1)^n(\mathcal{G}(q))^n(v_1)+\sum_{i=0}^{n-1}(-1)^i(\mathcal{G}(q))^i(f).
$$
As $\mathcal{G}(f) \in C([0,T];D(A^{\beta}))$, by (\ref{equation20}) and (\ref{equation22}), we obtain
$$
\|v_1\|_{C([0,T];D(A^{\beta}))} \leq\left\|(\mathcal{G}(q))^n(v_1)\right\|_{C([0,T];D(A^{\beta}))}
$$
$$
+\sum_{i=0}^{n-1}\left\|(\mathcal{G}(q))^i(f)\right\|_{C([0,T];D(A^{\beta}))}
$$
$$
\leq \frac{(CT^{\alpha}\|q\|_{C[0,T]}\Gamma(\alpha ))^n}{\Gamma(n \alpha +1)}\|v_1\|_{C([0,T];D(A^{\beta}))}
$$
$$
+\sum_{i=0}^{n-1} \frac{(CT^{\alpha}\|q\|_{C[0,T]}\Gamma(\alpha ))^i}{\Gamma(i \alpha +1)}\|\mathcal{G} f\|_{C([0,T];D(A^{\beta}))}
$$
$$
\leq \frac{(CT^{\alpha}\|q\|_{C[0,T]}\Gamma(\alpha ))^n}{\Gamma(n \alpha +1)} \|v_1\|_{C([0,T];D(A^{\beta}))}
$$
$$
+ \sum_{i=0}^{n-1} \frac{(CT^{\alpha}\|q\|_{C[0,T]}\Gamma(\alpha ))^i}{\Gamma(i \alpha +1)} T^\alpha\|f\|_{C\left([0,T];D\left(A^{\beta}\right)\right)}
$$
and by taking sufficiently large $k \in \mathbb{N}$, we get
\begin{equation}\label{equation23}
\|v_1\|_{C([0,T];D(A^{\beta}))} \leq C E_{\alpha, 1}\left(\Gamma(\alpha ) T^\alpha\|q\|_{C[0, T]}\right)\|f\|_{C\left([0,T]; D\left(A^{\beta}\right)\right)}
\end{equation}
with $C$ depending on $T,\,\, \alpha$.
By (\ref{equation23}), for all $t \in[0, T]$, we have $v \in D(A^{\beta})$ with
$$
A^{\beta} v(t) =\int_0^t A^{\beta} Y(t-s) q(s) v(s) d s
 +\int_0^t A^{\beta} Y(t-s) f(s) d s
$$
and by (\ref{equation8}), we have
$$
\|A^{\beta} Y(t)\| \leq C t^{\alpha-1}.
$$
The mapping $t \mapsto A^{\beta} Y(t)$ belongs to $C([0, T];H)$. Thus
$$
\|A^{\beta} v_1(t)\| \leq\left\|\int_0^t A^{\beta} Y(t-s) q(s) v_1(s) d s\right\|+  \left\|\int_0^t A^{\beta} Y(t-s) f(s) d s\right\|
$$
$$
\leq C \int_0^t(t-s)^{\alpha -1}\left(\|v_1(s)\|_{D(A^{\beta})}+\|f(s)\|_{D\left(A^{\beta}\right)}\right) d s
$$
\begin{equation}\label{equation24}
\leq \frac{C}{\alpha} t^{\alpha }\left(\|v_1\|_{C[0,t];D(A^{\beta}))}+\|f\|_{C\left([0,t];D\left(A^{\beta}\right)\right)}\right).
\end{equation}
Therefore, we have
\begin{equation}\label{equation25}
\|A^{\beta} v_1\|_{C([0,T];H)} \leq C T^{\alpha}\|f\|_{C\left([0,T] ; D\left(A^{\beta}\right)\right)}.
\end{equation}

By the original equation $\partial_{0+,t}^{\alpha} v=-A v-q v+f$, summing (\ref{equation21}), (\ref{equation23}) and (\ref{equation25}), we have $\partial_{0+,t}^{\alpha} v \in C([0, T] ; A^{\beta})$ with the estimate
\begin{equation*}
\left\|\partial_{0+,t}^{\alpha} v\right\|_{C([0,T];A^{\beta})}  \leq C\|f\|_{C\left([0,T];D\left(A^{\beta}\right)\right)}
\end{equation*}
\begin{equation*}
+\|q v\|_{C([0,T];A^{\beta})}
 +\|f\|_{C([0,T];A^{\beta})} \leq C\|f\|_{C\left([0, T] ; D\left(A^{\beta}\right)\right)}.
\end{equation*}

The Theorem \ref{theorem222} is proved.
\end{proof}

Now, similar to problem (\ref{equation4}), we will investigate problem (\ref{equation5}).

 The following theorem holds:
\begin{theorem}\label{theorem333}
Let $\varphi \in D\left(A^{2\beta}\right)$, $\psi \in D\left(A^{2\beta}\right)$ for some $\beta \in(0,1)$ and $q(t) \in C[0, T]$. Then the problem (\ref{equation5}) exists a unique solution $v_2 \in C([0,T];D(A^{\beta}))$ satisfying
$$
A^{\beta} v_2 \in C([0,T];H), \quad \partial_{0+,t}^{\alpha} v_2 \in C([0,T];H).
$$
Furthermore, there exists a constant $c>0$ depending on $\alpha, T, \beta$ and $\|q\|_{C[0, T]}$ such that
\begin{equation}\label{equation26}
\|A^{\beta} v_2\|_{C([0,T];H)}+\left\|\partial_{0+,t}^{\alpha} v_2\right\|_{C([0,T];H)} \leq C\left(\|\varphi\|_{D(A^{2\beta})}+T\|\psi\|_{D(A^{2\beta})}\right)
\end{equation}
and we have
\begin{equation}\label{equation27}
v_2(t)=Z_1(t) \varphi+Z_2(t) \psi-\mathcal{G}(q)(w)(t),
\end{equation}
where
$$
Z_1(t) \varphi=\sum_{k=1}^{\infty}\left(\varphi, e_k\right) E_{\alpha, 1}\left(-\lambda_k^{\beta} t^\alpha\right) e_k,
$$
$$
Z_2(t) \psi=t\sum_{k=1}^{\infty}\left(\psi, e_k\right) E_{\alpha, 1}\left(-\lambda_k^{\beta} t^\alpha\right) e_k
$$
in $C([0,T]; D(A^{\beta}))$ and the operator $\mathcal{G}$ defined in (\ref{equation18}).
\end{theorem}

Proof. We split the solution $w$ of (\ref{equation5}) into $v_2=W+\nu$, where $W$ satisfies
\begin{equation}\label{equation28}
\left\{\begin{array}{l}
\partial_{0+,t}^{\alpha}W(t)+A^{\beta}W(t)+q(t)W(t)=f_0(t), \quad 0<t \leq T, \\
W(t)\big|_{t=0}=0,\quad W_{t}(t)\big|_{t=0}=0.
\end{array}\right.
\end{equation}
with $f_0(t)=-A^{\beta} \nu-q(t) \nu$, $\nu=\varphi+t\psi$. By the conditions of Theorem \ref{theorem333}, we have $f_0(t) \in C\left([0, T] ; D\left(A^{\beta}\right)\right)$, and the estimate
\begin{equation*}
\left\|f_0\right\|^2_{C\left([0, T] ; D\left(A^{\beta}\right)\right)}=\max\limits_{0\le t \le T}\sum_{k=1}^{\infty}\lambda_k^{2\beta}\left|\left(A^{\beta} \nu+q(t) \nu, e_k\right) \right|^2
\end{equation*}
$$
=\max\limits_{0\le t \le T}\sum_{k=1}^{\infty}\lambda_k^{2\beta}\left(\left|\lambda_k^{\beta}\left(\nu,e_k\right) +q(t)\left(\nu,e_k\right)\right)\right|^2
$$
$$
\leq 2\max\limits_{0\le t \le T}\sum_{k=1}^{\infty}\lambda_k^{2\beta}\left(\lambda_k^{2\beta} +\|q(t)\|_{C[0,T]}\right)\left|\left(\nu,e_k\right)\right|^2 \leq C^2 \|\nu\|^2_{C([0,T];D(A^{2\beta})}
$$
or
\begin{equation}\label{equation29}
\left\|f_0\right\|_{C\left([0, T] ; D\left(A^{\beta}\right)\right)}  \leq C \|\nu\|_{C([0,T];D(A^{2\beta}))},
\end{equation}
where
$$
\|\nu\|_{C([0,T];D(A^{2\beta}))}\leq  \|\varphi\|_{D(A^{2\beta})}+T\|\psi\|_{D(A^{2\beta})}.
$$

Moreover, by Theorem \ref{theorem222}, the problem (\ref{equation28}) exists a unique solution $W \in C([0,T];D(A^{\beta}))$ satisfying
$$
A^{\beta} W \in C([0,T];H) \quad \text { and } \quad \partial_{0+,t}^{\alpha} W \in C([0, T] ; H)
$$
and the estimate
$$
\left\|\partial_{0+,t}^{\alpha} W\right\|_{C([0,T];H)}+\|A^{\beta} W\|_{C([0, T] ; H)} \leq
$$$$
C\left\|f_0\right\|_{C\left([0,T];D\left(A^{\beta}\right)\right)} \leq C\left(\|\varphi\|_{D(A^{2\beta})}+T\|\psi\|_{D(A^{2\beta})}\right).
$$
Therefore, the problem (\ref{equation5}) admits a unique solution $w=W+\nu \in C([0,T];D(A^{\beta}))$ satisfying
$$
A^{\beta} w \in C([0,T];H) \quad \text { and } \quad \partial_{0+,t}^{\alpha} w \in C([0,T];H)
$$
and the estimate (\ref{equation26}) holds.
Therefore, we have established the existence, uniqueness, and regularity of the solution for the direct problem.

\begin{theorem}\label{theorem44}
  Let $\varphi \in D\left(A^{2\beta}\right),$ $\psi \in D\left(A^{2\beta}\right)$ and $f \in C\left([0,T];D\left(A^{\beta}\right)\right)$ for some $\beta \in(0,1)$, and $q \in$ $C[0,T]$. Then there exists a unique solution $u \in C([0,T];D(A^{\beta}))$ to (\ref{equation1})-(\ref{equation2}) such that $\partial_{0+,t}^{\alpha} u \in C([0, T];H)$. Moreover there exists a constant $c>0$ such that
\begin{equation*}\label{equation30}
\|u\|_{C([0, T] ; D(A^{\beta}))} \leq C E_{\alpha, 1}\left(\Gamma(\alpha) T\|q\|_{C[0, T]}\right)
\end{equation*}
\begin{equation}\label{equation30}
\times\left[\|\varphi\|_{D\left(A^{2\beta}\right)}+\|\psi\|_{D\left(A^{2\beta}\right)}+\|f\|_{C\left([0, T] ; D\left(A^{\beta}\right)\right)}\right],
\end{equation}
and we get
\begin{equation}\label{equation31}
u(t)=Z_1(t) \varphi+Z_2(t) \psi+\mathcal{G}(f)(t)-(\mathcal{G}(q))(u)(t),
\end{equation}
where $\mathcal{G}$ is defined by (\ref{equation18}).
\end{theorem}

The continuous dependence of the solution to problem (\ref{equation1})-(\ref{equation2}) on the data is given
by the following theorem.

\begin{theorem}\label{theorem55}
Under the same conditions as Theorem \ref{theorem44}, the solution of the direct problem (\ref{equation1})-(\ref{equation2}) depends continuously on the given data, that is
$$
\|u-\hat{u}\|_{C([0, T] ; D(A^\beta))} \leq C\Bigg[\|\varphi-\hat{\varphi}\|_{D\left(A^{2\beta}\right)}
$$
\begin{equation}\label{equation32}
+\|\psi-\hat{\psi}\|_{D\left(A^{2\beta}\right)}+\|q-\hat{q}\|_{C[0, T]}+\|f-\hat{f}\|_{C\left([0, T] ; D\left(A^\beta\right)\right)}\Bigg],
\end{equation}
where $C>0$ depending on $\alpha, \,\, T$ and $\|q\|_{C[0, T]}$.
\end{theorem}

\begin{proof}
Let $u$ and $\hat{u}$ be the solutions to the direct problem (\ref{equation1})-(\ref{equation2}), respectively to the functions $\{q, f, \varphi, \psi\}$ and $\{\hat{q}, \hat{f}, \hat{\varphi}, \hat{\psi}\}$. Using $u(t)=v(t)+w(t)$, we have
$$
|u(t)-\hat{u}(t)| \leq|v(t)-\hat{v}(t)|+|w(t)-\hat{w}(t)|,
$$
where $\{v, w\}$ and $\{\hat{v}, \hat{w}\}$ corresponding to the data $\{q, f, \varphi, \psi\}$ and $\{\hat{q}, \hat{f}, \hat{\varphi}, \hat{\psi}\}$, respectively. Using (\ref{equation9}), we obtain
$$
\|v(t)-\hat{v}(t)\|_{D(A^{\beta})} \leq\left\|\int_0^t A^{\beta} Y(t-s)[q(s) v(s)-\hat{q}(s) \hat{v}(s)] d s\right\|
$$
$$
+\left\|\int_0^t A^{\beta} Y(t-s)(f(s)-\hat{f}(s)) d s\right\|
$$
$$
\leq C \int_0^t(t-s)^{\alpha-1}\left|q(s)-\hat{q}(s)\right| \left\| v(s)\right\|_{D(A^{\beta})} d s
$$$$
+C \int_0^t(t-s)^{\alpha-1}\left\| v(s)-\hat{v}(s) \right\|_{D(A^{\beta})} \left|\hat{q}(s)\right| d s
$$
$$
+C \int_0^t(t-s)^{\alpha-1}\|f(s)-\hat{f}(s)\|_{D\left(A^{\beta}\right)} d s\leq C \frac{t^{\alpha}}{\alpha}\|v\|_{C([0,T];D(A^{\beta}))}\|q-\hat{q}\|_{C[0,T]}
$$
$$
+C \frac{t^{\alpha}}{\alpha}\|f-\hat{f}\|_{C\left([0, T] ; D\left(A^{\beta}\right)\right)}+C\|\hat{q}\|_{C[0, T]} \int_0^t(t-s)^{\alpha-1}\|v(s)-\hat{v}(s)\|_{D(A^{\beta})} d s.
$$

Then, according to the Grönwall's inequality given in Lemma 2, we obtain:
\begin{equation}\label{equation33}
\|v(t)-\hat{v}(t)\|_{D(A^{\beta})} \leq C\left[\|q-\hat{q}\|_{C[0, T]}+\|f-\hat{f}\|_{C\left([0, T] ; D\left(A^{\beta}\right)\right)}\right], \end{equation}
where $C>0$ depending on $\alpha, \,\, T,\,\,  \|\hat{q}\|_{C[0, T]}$.
The same arguments as applied to (\ref{equation31}) before lead to
\begin{equation}\label{equation34}
\|w(t)-\hat{w}(t)\|_{D(A^{\beta})} \leq C\left[\|\varphi-\hat{\varphi}\|_{D\left(A^{2\beta}\right)}+\|\psi-\hat{\psi}\|_{D\left(A^{2\beta}\right)}+\|q-\hat{q}\|_{C[0, T]}\right],
\end{equation}
where $C>0$ depending on $\alpha, \,\, T$ and $\|\hat{q}\|_{C[0, T]}$.
Finally, from (\ref{equation33}), (\ref{equation34}), we get the desired estimate (\ref{equation32}). The Theorem \ref{theorem55} is proved.

\end{proof}

\section*{Investigation of the inverse problem (\ref{equation1}) - (\ref{equation3})}

The following main result holds for the inverse problem.

\begin{theorem}\label{theorem44444}
Let (I1)-(I3) be held. Then the problem of finding a solution of (\ref{equation1})-(\ref{equation3}) is equivalent to the problem of determining the function $q(t) \in C[0,T]$ satisfying
\begin{equation}\label{equation35}
q(t)=\frac{1}{\mu(t)}\left(\Phi[f](t)-\partial_{0+,t}^{\alpha} \mu(t)-\Phi[A^{\beta} u](t)\right),
\end{equation}
where
$$
A^{\beta}u(t)=A^{\beta}Z_1(t) \varphi+A^{\beta}Z_2(t) \psi
$$
\begin{equation}\label{equation36}
+\int_0^t A^{\beta}Y(t-s) f(s) d s-\int_0^t A^{\beta}Y(t-s) q(s) u(s) d s.
\end{equation}
\end{theorem}

On the other hand, if (\ref{equation35}) has a solution and the technical condition (I2)-(I4) holds, then there exists a solution to the inverse problem (\ref{equation1})-(\ref{equation3}).

\begin{proof}

We split the proof into two steps.
\textit{Step 1.} Suppose that the problem (\ref{equation1})-(\ref{equation3}) has a solution $q(t) \in C[0,T]$. Taking into account (I2), and apply $\Phi$ to the equation of (\ref{equation1}) yields
\begin{equation}\label{equation37}
\partial_{0+,t}^{\alpha} \Phi[u](t)+\Phi[A^{\beta} u](t)+q(t) \Phi[u](t)=\Phi[f](t).
\end{equation}
Taking into account conditions  (\ref{equation2}),(\ref{equation3}) and I2)-I4),  we obtain
$$
\partial_{0+,t}^{\alpha} \mu(t)+\Phi[A^{\beta} u](t)+q(t)\mu(t)=\Phi[f](t).
$$
From the above relation, we obtain the following equation for $q(t)$
$$
q(t)=\frac{1}{\mu(t)}\left(\Phi[f](t)-\partial_{0+,t}^{\alpha} \mu(t)-\Phi[A^{\beta} u](t)\right).
$$
From this, we obtain the relation (\ref{equation35}) of Theorem \ref{theorem44444}.

\textit{Step 2.} Suppose now that $q \in C[0, T]$ satisfy (\ref{equation35}). In order to prove that $q$ is the solution to the inverse problem (\ref{equation1})-(\ref{equation3}), it suffices to show that (\ref{equation3}). By the equation (\ref{equation1}), we have (\ref{equation37}). Together with (\ref{equation35}) and (I3), we obtain that $y(t):=$ $\Phi[u](t)-\mu(t)$ satisfies
\begin{equation}\label{equation38}
\left\{\begin{array}{l}
\partial_{0+,t}^{\alpha} y(t)+q(t) y(t)=0, \quad t \in(0, T], \\
y(0)=0,\,\, y'(0)=0.
\end{array}\right.
\end{equation}
Therefore, we have (see \cite[p. 199]{forhtur1})
$$
y(t)=-\frac{1}{\Gamma(\alpha)} \int_0^t(t-s)^{\alpha-1} q(s) y(s) d s.
$$
Then for $q(t) \in C[0,T]$, we have
$$
\|y\|_{C[0, t]} \leq \frac{1}{\Gamma(\alpha)}\|q\|_{C[0, T]} \int_0^t(t-s)^{\alpha-1}\|y\|_{C[0, s]} d s
$$
for all $t \in[0, T]$. Hence, According to Lemma \ref{lemma1}, we have $\|y\|_{C[0, t]}=0$ for all $t \in[0, T]$, which implies $\Phi[u](t)=\mu(t)$ on $[0, T]$. The theorem \ref{theorem44444} is proved.
\end{proof}

\textbf{Remark} 
(I3) is the consistency condition for our problem (\ref{equation1})-(\ref{equation3}), which guarantees that the inverse problem (\ref{equation1})-(\ref{equation3}) is equivalent to (\ref{equation35}).

\begin{theorem}\label{theorem555}
Under hypotheses (I1)-(I4), there exists a sufficiently small $T > 0$ such that the inverse problem (\ref{equation1})-(\ref{equation3}) has a unique solution $q(t) \in C[0,T]$.
\end{theorem}

\begin{proof}
  Let us define the following operator
\begin{equation}\label{equation39}
\begin{cases}\mathcal{Q}: C[0, T] & \rightarrow \quad C[0, T], \\
q \rightarrow \mathcal{Q}(q): &  \quad t \mapsto \frac{1}{\mu(t)}\left(\Phi[f](t)-\partial_{0+,t}^{\alpha} \mu(t)-\Phi[A^{\beta} u](t)\right).
\end{cases}
\end{equation}

To prove that the operator $\mathcal{Q}$ admits a fixed point, start by showing that $\mathcal{Q}$ maps a certain closed convex set into itself, in the space $C[0, T]$ equipped with the sup norm.

First, we will show that there exists a positive constant $T_1>0$ such that for any $T \in (0,T_1]$, there is a radius $R>0$ such that the closed convex ball
$$
B_R=\left\{q \in C[0, T]:\|q\|_{C[0, T]} \leq R\right\}
$$
is stable by the operator $\mathcal{Q}$; that is, $\mathcal{Q}(B_R) \subset B_R$.
According to the definition of the operator $A^{\beta}$ and linearity, and due to (I4), we have
$$
\Phi[A^{\beta} u](t)=\sum_{k=1}^{\infty} \lambda_k^{\beta}\left(u, e_k\right) \Phi\left[e_k\right]
$$
and by the Hölder's inequality, we obtain
$$
|\Phi[A^{\beta} u](t)| \leq\left(\sum_{k=1}^{\infty} \Phi^2\left[e_k\right]\right)^{1 / 2}\left(\sum_{k=1}^{\infty}\left(\lambda_k^{\beta}\left(u, e_k\right)\right)^2\right)^{1 / 2}=c\|u(t)\|_{D(A^{\beta})}
$$
or
\begin{equation}\label{equation40}
\|\Phi[A^{\beta} u]\|_{L^{\infty}(0, T)} \leq \widetilde{C}\|u\|_{C([0,T];D(A^{\beta}))}.
\end{equation}
Then, for any $q(t) \in B_R$ and from the linearity of $\Phi[\cdot]$ and due to condition (I4), we have
\begin{equation*}
 |\mathcal{Q}(q)(t)|=\left|\frac{1}{\mu(t)}\left(\Phi[f](t)-\partial_{0+,t}^{\alpha} \mu(t)-\Phi[A^{\beta} u](t)\right)\right|   \leq \frac{1}{\mu_0}\Bigg(|\Phi[f](t)|+\left|\partial_{0+,t}^{\alpha} \mu(t)\right|
 \end{equation*}
\begin{equation*}
 +|\Phi[A^{\beta} u](t)|\Bigg)
 \leq \frac{1}{\mu_0}\left(\left|\sum_{k=1}^{\infty}\left(f(t), e_k\right) \Phi\left[e_k\right]\right|+\left\|\partial_{0+,t}^{\alpha} \mu\right\|_{C[0, T]} +\widetilde{C}\|u(t)\|_{D(A^{\beta})}\right)
\end{equation*}
\begin{equation*}
\leq \frac{1}{\mu_0}\left[\widetilde{C}\|f(t)\|_{D\left(A^\beta\right)}+\left\|\partial_{0+,t}^{\alpha} \mu\right\|_{C[0, T]}\right.
\end{equation*}
\begin{equation*}
\left.+\widetilde{C} E_{\alpha, 1}(\Gamma(\alpha ) T R)\left(\|\varphi\|_{D\left(A^{2\beta}\right)}+
\|\psi\|_{D\left(A^{2\beta}\right)}+\|f\|_{C\left([0,T];D\left(A^\beta\right)\right)}\right)\right].
\end{equation*}
Then we can choose sufficiently small $T_1$ such that
\begin{equation*} \frac{1}{\mu_0}\left[\widetilde{C}\|f(t)\|_{D\left(A^\beta\right)}+\left\|\partial_{0+,t}^{\alpha} \mu\right\|_{C[0, T]}\right.
\end{equation*}
\begin{equation*}
\left.+\widetilde{C} E_{\alpha, 1}(\Gamma(\alpha ) T R)\left(\|\varphi\|_{D\left(A^{2\beta}\right)}+
\|\psi\|_{D\left(A^{2\beta}\right)}+\|f\|_{C\left([0,T];D\left(A^{\beta}\right)\right)}\right)\right]\leq R
\end{equation*}
for all $T<T_1$ to obtain
\begin{equation}\label{equation41}
\|\mathcal{Q}(q)\|_{C[0,T]} \leq R.
\end{equation}
Now, we verify the second condition of the fixed point theorem. Let $q(t), \hat{q}(t) \in \mathrm{B}$ be given. Then for the difference of the operators we have
$$
\mathcal{Q}(q)(t)-\mathcal{Q}(\hat{q})(t)=-\frac{1}{\mu(t)}(\Phi[A^{\beta} u](t)-\Phi[A^{\beta} \hat{u}](t)).
$$
By the linearity of $\Phi[\cdot]$ and (I4), we obtain
$$
|\mathcal{Q}(q)(t)-\mathcal{Q}(\hat{q})(t)| \leq \frac{1}{\mu_0}|\Phi[A^{\beta}(u-\hat{u})](t)| \leq \widetilde{C} \frac{1}{\mu_0}\|u(t)-\hat{u}(t)\|_{D(A^{\beta})}.
$$

Then, by Theorem \ref{theorem55}, we have
\begin{equation*}
\|\mathcal{Q}(q)-\mathcal{Q}(\hat{q})\|_{C[0, T]} \leq  C \widetilde{C} \frac{1}{\mu_0}\|q-\hat{q}\|_{C[0, T]},
\end{equation*}
where $C$ is the same as (\ref{equation32}). Therefore, we can choose sufficiently small $T_2$ such that
\begin{equation*}
C(T) \widetilde{C}(T) \frac{1}{\mu_0}:=r<1
\end{equation*}
for all $T \in\left(0,T_2\right]$ to obtain
\begin{equation}\label{equation42}
\|\mathcal{Q}(q)-\mathcal{Q}(\hat{q})\|_{C[0, T]} \leq r\|q-\hat{q}\|_{C[0, T]}.
\end{equation}

Estimates (\ref{equation41}) and (\ref{equation42}) show that $\mathcal{Q}$ is a contraction map on $B_R$ for all $T \in(0,T_0]$, if we choose $T_0 \leq \min \left\{T_1, T_2\right\}$.

The theorem \ref{theorem555} is proved.
\end{proof}

Now we will prove the theorem about the stability of the solution of the inverse problem.

\begin{theorem}\label{theorem666}
Let conditions (I1)-(I4) be fulfilled and $u_i$ be the solution of (\ref{equation1})-(\ref{equation3}) for $q=q_i \in C[0, T]$ with $\|q\|_{C[0, T]} \leq R \, (i=1,2)$. Assume that there exists $\kappa>0$ such that
\begin{equation}\label{equation43}
\left|\Phi\left[u_2\right](t)\right| \geq \kappa^{-1}>0, \quad \text { for all } \quad t \in[0, T].
\end{equation}
Then there exists a constant $\tilde{C}>0$ depending on $R, T, \alpha$ and $\mu_0$ such that
$$
\tilde{C}^{-1}\left\|\partial_{0+,t}^{\alpha}\left(\Phi\left[u_1\right]-\Phi\left[u_2\right]\right)\right\|_{C[0, T]}
$$
\begin{equation}\label{equation44}
\leq\left\|q_1-q_2\right\|_{C[0, T]} \leq \tilde{C}\left\|\partial_{0+,t}^{\alpha}\left(\Phi\left[u_1\right]-\Phi\left[u_2\right]\right)\right\|_{C[0, T]}
\end{equation}
and
\begin{equation}\label{equation45}
\left\|u_1-u_2\right\|_{C([0, T] ; D(A^{\beta}))} \leq \tilde{C}\left\|q_1-q_2\right\|_{C[0, T]}.
\end{equation}
\end{theorem}

\begin{proof}
We assume that $u_i$ are two solutions to (\ref{equation1})-(\ref{equation2}) corresponding to $q=q_i(i=1,2)$. Let $u=u_1-u_2$ and $q=q_2-q_1$. Then $u$ satisfies
\begin{equation}\label{equation46}
\left\{\begin{array}{l}
\partial_{0+,t}^{\alpha} u(t)+A^{\beta} u(t)+q_1(t) u(t)=q(t) u_2(t), \quad t \in(0, T], \\
u(0)=0,\,\, u_t(0)=0.
\end{array}\right.
\end{equation}

According to Theorem \ref{theorem44}, $u(t)$ is given by
$$
u(t)=\mathcal{G}(q)\left(u_2\right)-\mathcal{G}\left(q_1\right)(u).
$$

Now, we get the upper estimate for $\|u(t)\|_{D(A)}$. Similar to the argument of (\ref{equation24}), we have
$$
\|u(t)\|_{D(A^{\beta})} \leq C\left[\int_0^t(t-s)^{\alpha -1}\|u(s)\|_{D(A^{\beta})} d s+\int_0^t(t-s)^{\alpha -1}|q(s)| d s\right],
$$
where $C>0$ are depend on $\alpha,$ $T$, $\|f\|_{C\left([0, T] ; D\left(A^{\beta}\right)\right)}$, $\|\varphi\|_{D\left(A^{2\beta}\right)}$,  $\|\psi\|_{D\left(A^{2\beta}\right)}$, and $\left\|q_1\right\|_{C[0, T]}$. Then, according to Lemma \ref{lemma1}, we have
$$
\|u(t)\|_{D(A^{\beta})} \leq C Z_{\alpha, 1}(C \Gamma(\alpha) t) \int_0^t(t-s)^{\alpha-1}|q(s)| d s
$$
or
$$
\|u\|_{C([0, T] ; D(A^{\beta}))} \leq C\frac{T^{\alpha }}{\alpha } Z_{\alpha, 1}(C \Gamma(\alpha) T)\|q\|_{C[0, T]}.
$$

That is, (\ref{equation45}) is true.\\
Applying $\Phi[\cdot]$ in equation (\ref{equation46}), the procedure yields
\begin{equation}\label{equation47}
\Phi\left[u_2\right](t) q(t)=\partial_{0+,t}^{\alpha} \Phi[u](t)+\Phi[A^{\beta} u](t)+q_1(t) \Phi[u](t), \quad t \in(0, T).
\end{equation}

By performing calculations like those in equations (\ref{equation30}) and  (\ref{equation31}), and by (\ref{equation43}), we can obtain
$$
|q(t)| \leq \kappa\left|\partial_{0+,t}^{\alpha} \Phi[u](t)+\Phi[A^{\beta} u](t)+q_1(t) \Phi[u](t)\right|
$$
$$
\leq \kappa\left\|\partial_{0+,t}^{\alpha} \Phi[u]\right\|_{C[0, T]}+C \kappa \int_0^t(t-s)^{\alpha -1}|q(s)| d s, \quad t \in(0,T].
$$
Again, using Lemma \ref{lemma1}, we obtain:
$$
\|q\|_{C[0, T]} \leq C\left\|\partial_{0+,t}^{\alpha} \Phi[u]\right\|_{C[0, T]}
$$
and from this, we obtain the right hand side of (\ref{equation44}). On the other hand, from (\ref{equation47}), we obtain
\begin{equation*}
\left|\partial_{0+,t}^{\alpha} \Phi[u](t)\right| \leq\left|\Phi\left[u_2\right](t) q(t)\right|+|\Phi[A^{\beta} u](t)|+\left|q_1(t) \Phi[u](t)\right|
\end{equation*}
\begin{equation*}
\leq C|q(t)|\left\|u_2(t)\right\|_{D(A^{\beta})}+C \int_0^t(t-s)^{\alpha  -1}|q(s)| d s
\end{equation*}
\begin{equation*}
 \leq C|q(t)|\left\|u_2\right\|_{C([0, T] ; D(^{\beta}))}+C \frac{T^{\alpha }}{\alpha }\|q\|_{C[0, T]}.
\end{equation*}
Therefore, we get
$$
\left\|\partial_{0+,t}^{\alpha} \Phi[u]\right\|_{C[0, T]} \leq C\left(\left\|u_2\right\|_{C([0, T] ; D(A^{\beta}))}+T^{\alpha }\right)\|q\|_{C[0, T]}.
$$
The theorem \ref{theorem666} is proved.
\end{proof}

\section*{Conclusion}
In this article, we studied the inverse problem of determining the unknown coefficient in the fractional-order abstract diffusion-wave equation. We initially analyzed the correctness of the direct problem.  Then we determined the relationship between the equivalence between the inverse problem and a corresponding integral equation. Using the properties of the solution of the direct problem, we considered the inverse problem in more detail.   Using the theorem a fixed point in Banach space, we obtained the local existence, uniqueness, and stability of the solution of the inverse problem.

\end{document}